\documentclass[final]{article}
\usepackage[colorlinks = true, allcolors = blue]{hyperref}
\usepackage[leqno]{amsmath}
\usepackage{amssymb, amsfonts, amsthm, stmaryrd, graphicx, 
cleveref, physics, tikz, tcolorbox, subcaption, float, nicefrac, doi}
\usepackage[numbers,sort]{natbib}
\usepackage[english]{babel}
\usepackage[letterpaper, top = 2cm,bottom = 2cm, left = 3cm, right = 3cm, marginparwidth = 1.75cm]{geometry}
\newtheorem{theorem}{Theorem}[section]

\newtheorem{lemma}[theorem]{Lemma}
\newtheorem{proposition}[theorem]{Proposition}

\newtheorem{definition}{Definition}[section]

\theoremstyle{remark}
\newtheorem{remark}{Remark}[section]

\newcommand{\R}[0]{\mathbb{R}}

\DeclareMathOperator{\ran}{\operatorname{\textnormal{ran}}}

\newcommand{\DECd}{\textnormal{d}^\textnormal{DEC}}
\newcommand{\DECD}{\textnormal{D}^\textnormal{DEC}}
\newcommand{\DD}{\textnormal{D}}
\newcommand{\DECdelta}{\delta^\textnormal{DEC}}
\newcommand{\inner}[2]{\left\langle #1, #2 \right\rangle}
\newcommand{\innerc}[2]{\left\llbracket #1, #2 \right\rrbracket}
\newcommand{\normm}[1]{{\left\vert\kern-0.25ex\left\vert\kern-0.25ex\left\vert #1 
    \right\vert\kern-0.25ex\right\vert\kern-0.25ex\right\vert}}
\newcommand{\LambdaForms}[1][]{{\Lambda^{#1}(\Omega)}}

\newcommand{\ccstar}{
    \tikz[scale=0.1,baseline=-0.25em]{
        \draw[line width=0.5pt, join=round]
          (90:1) -- (234:1) -- (18:1) -- (162:1) -- (306:1) -- cycle;
    }
}

\newcommand{\rev}[1]{#1}
\crefname{theorem}{theorem}{theorems}
\Crefname{theorem}{Theorem}{Theorems}

\crefname{lemma}{lemma}{lemmas}
\Crefname{lemma}{Lemma}{Lemmas}

\crefname{corollary}{corollary}{corollaries}
\Crefname{corollary}{Corollary}{Corollaries}

\crefname{proposition}{proposition}{propositions}
\Crefname{proposition}{Proposition}{Propositions}

\crefname{claim}{claim}{claims}
\Crefname{claim}{Claim}{Claims}

\crefname{definition}{definition}{definitions}
\Crefname{definition}{Definition}{Definitions}

\crefname{example}{example}{examples}
\Crefname{example}{Example}{Examples}

\crefname{remark}{remark}{remarks}
\Crefname{remark}{Remark}{Remarks}

\crefname{assumption}{assumption}{assumptions}
\Crefname{assumption}{Assumption}{Assumptions}

\crefname{Rule}{rule}{rules}
\Crefname{Rule}{Rule}{Rules}
\newcommand{\ipop}{\rev{\mathcal{R}}}

\title{Convergence of Discrete Exterior Calculus for the Hodge-Dirac Operator}
\author{Radovan Dabeti\'c\thanks{D-MATH, ETH Zurich, CH-8092 Zürich, Switzerland,
    \texttt{rdabetic@ethz.ch}} \and
  Ralf Hiptmair\thanks{SAM, D-MATH, ETH Zurich, CH-8092 Zürich, Switzerland,
    \texttt{hiptmair@sam.math.ethz.ch}}} 

\begin{document}
\maketitle

\begin{abstract}
  A short proof of convergence for the discretization of the Hodge-Dirac operator in the
  framework of discrete exterior calculus (DEC) is provided using the techniques established
  in $[$\textsc{Johnny Guzmán and Pratyush Potu}, {\em A Framework for Analysis of DEC
    Approximations to Hodge-Laplacian Problems using Generalized Whitney Forms},
  \href{https://arxiv.org/abs/2505.08934}{arXiv Preprint 2505.08934}, 2025$]$.
\end{abstract}

\section{Introduction}
\label{sec:intro}

There are two fundamental operators associated with $L^{2}$-de Rham complexes, the
Hodge-Laplacian and the Hodge-Dirac operator. \rev{In this work we are concerned with
  homogeneous boundary value problems for the latter on bounded domains
  $\Omega$ of Euclidean space $\mathbb{R}^{n}$. Similar boundary value problems occur
  in models of particle physics like, for instance, the MIT bag model \cite{JOH75,ALR16}
  for ${n=3}$, and as cell problems in electronic structure models for graphene
  \cite{CGP09} for $n=2$. Refer to the \cite[Section~1]{BHM20} for more references.}

\rev{Given a triangulation of the domain $\Omega$}, there are two fundamental ways to
discretize de Rham complexes. The first is a finite-element approach employing discrete
differential forms known as finite element exterior calculus (FEEC). The second has the
flavor of a finite-volume technique and is called discrete exterior calculus (DEC).

A comprehensive a priori convergence theory for the DEC approximation of Hodge-Laplacians
has recently been achieved by J.~Guzmán and P.~Potu in the breakthrough work
\cite{GUP25}. The present paper harnesses the novel techniques from \cite{GUP25} to
establish the convergence of the DEC discretization of the Hodge-Dirac operator associated
with the $L^{2}$ de Rham Hilbert complex on a bounded domain of ${n}$-dimensional
Euclidean space. This work supplements \cite{GUP25}, from where we borrow the bulk of our
notation, often without defining it again. \rev{Note that, however, the analysis presented
  in this paper is primarily conducted from the perspective of cochains, whereas the one
  in \cite{GUP25} more heavily relies on Whitney forms, lowest order discrete differential
  forms.}  We also refer to \cite{GUP25} for background information on FEEC and DEC and a
discussion of pertinent literature. The reader is advised to study \cite{GUP25} before
reading the present paper.

\section{The Hodge-Dirac Operator}
\label{sec:dirac}

Let $\Omega\subset\R^n$ be a bounded, Lipschitz, polytopal, and \emph{topologically trivial}
domain and we write $\LambdaForms[k]$ for the space of smooth $k$-forms thereon.
% \ footnote{The notation in this manuscript is largely adopted from \cite{GUP25}.}
As in \cite{GUP25} the exterior derivative operators are denoted by
$d^k: \LambdaForms[k]\to\LambdaForms[k + 1]$, $0\leq k<n$, the (Euclidean) Hodge star
operators by $\star_k$ and the codifferential operators by
$\delta_k := (-1)^k\star_{k - 1}^{-1} d^{n-k}\star_k: \LambdaForms[k]\to\LambdaForms[k -
1], k = 1, \dots, n$. The Hodge star operators induce inner products on $\LambdaForms[k]$.
\begin{definition}
  \label{def:1}
  The $L^2$ inner product on two $k$-forms $\omega$ and $\mu$ is given by 
  \begin{gather*}
    \inner{\omega}{\mu}_{L^2\LambdaForms[k]} := \int_\Omega\omega\wedge\star\mu.
  \end{gather*}
\end{definition}

We denote by $\LambdaForms := \bigoplus_{k = 0}^n \LambdaForms[k]$ the exterior
algebra of (smooth) differential forms on $\Omega$ and write 
\begin{gather}
  \label{ddd}
  \dd :=
  \begin{pmatrix}
    0 &  & \\
    d^0 & 0 & \\
    & d^1 & 0 & \\
    && \ddots & \ddots
  \end{pmatrix}, \quad \delta :=
  \begin{pmatrix}
    0 & \delta_1 & \\
    & 0 & \delta_2 \\
    & & 0 & \ddots \\
    && & \ddots
  \end{pmatrix}
\end{gather}
for the exterior derivative and codifferential on $\LambdaForms$. We equip
$\LambdaForms$ with the natural Hilbert space structure by combining the inner products
from \Cref{def:1}. For $\mathfrak{u} \equiv (u_0, \dots, u_n), \mathfrak{v} \equiv (v_0, \dots, v_n)\in\LambdaForms$
we set
\begin{gather*}
  \inner{\mathfrak{u}}{\mathfrak{v}}_{L^2\Lambda} := \sum_{k = 0}^n
  \inner{u_k}{v_k}_{L^2\Lambda}\;.
\end{gather*}
Write $L^2\LambdaForms:= \bigoplus_{k = 0}^n L^2\LambdaForms[k]$, where
$L^2\LambdaForms[k]$ is the space of square-integrable $k$-forms, i.e.\ $k$-forms with
coefficients in $L^2(\Omega)$.

Also refer to \cite[Section 6.2.6]{Arnold2018}, where Sobolev spaces of differential forms
are introduced.  Let
\begin{gather*}
  H\LambdaForms := \left\{ \mathfrak{u}\in L^2\LambdaForms: \dd\mathfrak{u}\in
    L^2\LambdaForms\right\}\;,
\end{gather*}
and define
$\mathring{V} := \bigoplus_{k = 0}^{n-1}\mathring{H}\LambdaForms[k] \oplus
L^2_*\LambdaForms[n]$, where
\begin{gather}
  L^2_*\LambdaForms[n] := \left\{v\in L^2\LambdaForms[n]:
  \int_\Omega v = 0 \right\}\;,
\end{gather}
and $\mathring{H}\LambdaForms[k]$ is the
space of functions in $H\LambdaForms[k]$ with vanishing trace on $\partial\Omega$,
see \cite[Section 6.2.6]{Arnold2018}. Also let $H^*\LambdaForms$ be the domain of
$\delta$, see also \cite[Section 6.2.6]{Arnold2018}.

\begin{definition}
  \label{def:D}
  The Hodge-Dirac operator is $\DD := \dd + \delta$ with domain of definition 
  $\mathcal{D}(\DD) := \mathring{H}\Lambda(\Omega)\cap H^*\Lambda(\Omega)$, where the domain of $\dd$ is
  $\mathring{H}\Lambda(\Omega)$ and that of $\delta$ is $H^*\Lambda(\Omega)$. 
\end{definition}

Taking the cue from \cite{Leopardi2016} we put the focus on the following boundary
value problem\footnote{Note that the boundary conditions are included implicitly in the
  domain of the operator.} for the Dirac operator: Given \rev{ $\mathfrak{f} = (f_0, \dots, f_n) \in L^2\LambdaForms$}, seek
$\mathfrak{u}\in\mathcal{D}(\DD)\cap(\ker\DD)^\perp, \mathfrak{p}\in\ker\DD$ such that
\begin{equation}\label{eq:strong_form}
  \DD\mathfrak{u} + \mathfrak{p} = \mathfrak{f}.
\end{equation}
Corollary 8 of \cite{Leopardi2016} tells us the well-posedness of the following weak form of
\eqref{eq:strong_form}: Given $\mathfrak{f}\in L^2\LambdaForms$, seek
$\mathfrak{u}\in\mathring{H}\LambdaForms, \mathfrak{p}\in\ker\DD$ such that
\begin{align}\label{eq:var}
  \begin{aligned}
    \inner{\dd\mathfrak{u}}{\mathfrak{v}}_{L^2\Lambda} +
    \inner{\mathfrak{u}}{\dd\mathfrak{v}}_{L^2\Lambda} +
    \inner{\mathfrak{p}}{\mathfrak{v}}_{L^2\Lambda} &
    = \inner{\mathfrak{f}}{\mathfrak{v}}_{L^2\Lambda} &&
    \forall\mathfrak{v}\in\mathring{H}\Lambda(\Omega), \\
    \inner{\mathfrak{u}}{\mathfrak{v}}_{L^2\Lambda} & = 0 &&
    \forall\mathfrak{v}\in\ker\DD.
  \end{aligned}
\end{align}
As we are working with a domain with trivial topology, $\ker\DD$ (the space of
harmonic forms) is trivial (see \cite[Section 4.3]{Arnold2018} for more information)
except for constant $n$-forms, i.e.\ $\ker\DD|_{\mathring{V}} = \{0\}$, so that we can
consider the following simpler problem: Given $\mathfrak{f}\in L^2\LambdaForms$ with
$\int_\Omega f_n = 0$\footnote{Or given the general case, we can recover
  $\mathfrak{p}$ by taking the mean of the $n$-form in $\mathfrak{f}$ and then
  subtract the mean to get a suitable right-hand side.}, seek
$\mathfrak{u}\in\mathring{V}$ such that
\begin{gather}
  \label{eq:3}
  \inner{\dd\mathfrak{u}}{\mathfrak{v}}_{L^2\Lambda} +
  \inner{\mathfrak{u}}{\dd\mathfrak{v}}_{L^2\Lambda} =
  \inner{\mathfrak{f}}{\mathfrak{v}}_{L^2\Lambda}
  \quad\forall\mathfrak{v}\in\mathring{V}.
\end{gather}

\rev{\begin{remark}
  \label{rem:dopf}
  For $n=3$ in physics the Dirac operator is usually expressed by means of Pauli matrices as a
  differential operator acting on functions $\Omega\to\mathbb{C}^{4}$. Identifying $\mathbb{C}^{4}$ with
  $\mathbb{R}^{8}$ in a suitable way, this is exactly the differential operator of
  \Cref{def:D} as is explained in \cite[Section~2]{CHR18} and \cite[Section~1.1]{FRI23}.
\end{remark}}

\section{DEC Discretization of the Hodge-Dirac Operator}

In order to discretize \eqref{eq:3} we rely on an oriented \emph{well-centered} simplicial
mesh\footnote{See \cite{Christiansen2011} for more information on cellular complexes.}
$\mathcal{T}$ of $\Omega$, that is, as stipulated by \cite[Definition 2.4.3]{Hirani2003}
the circumcenter of any simplex of $\mathcal{T}$ lies in its interior. Write
$\mathcal{T}^k$ for the $k$-cells of $\mathcal{T}$ and $C^k(\mathcal{T})$ the $k$-cochains
on $\mathcal{T}$. Furthermore, let $\tilde{\mathcal{T}}$ designate the (orthogonal)
\emph{dual mesh} of $\mathcal{T}$ and let $*\sigma\in\tilde{\mathcal{T}}$ for
$\sigma\in\mathcal{T}^k$ denote the dual cell of $\sigma$, see \cite[Section 3.2]{GUP25}.
\rev{Furthermore, as in \cite[Section~5.2]{GUP25}, we will have to make some assumptions
  on the dual mesh, namely, the dual mesh should be shape-regular in an
  appropriate sense, see \cite[Section~5.2]{GUP25}, and \cite[Assumption~5.5]{GUP25}
  should hold.}

We denote by $\partial^k: C^k(\mathcal{T}) \to C^{k+1}(\mathcal{T})$ the coboundary
operator, with respect to the coordinate basis $C^k(\mathcal{T})$ and
$C^{k+1}(\mathcal{T})$ represented by the incidence matrix of oriented $k$ and
$(k+1)$-cells. Let $\ccstar_{k}: C^k(\mathcal{T}) \to C^{n-k}(\tilde{\mathcal{T}})$,
$k = 0, \dots, n$, be the customary discrete Hodge stars introduced in discrete exterior
calculus, see \cite[Definition 3.6]{GUP25}. As in \cite[Definition 3.6]{GUP25}, let
$\ccstar_k^{-1}$ denote its inverse .

\begin{definition}[{\cite[Section 4.5]{GUP25}}]\label{def:dec_inner} For $\mathsf{u}, \mathsf{v} \in C^k(\mathcal{T})$,
  define the inner product
  \begin{gather}
    \label{eq:ipc}
    \innerc{\mathsf{u}}{\mathsf{v}}_k :=
    \sum_{\sigma\in\mathcal{T}^k}
    \frac{|*\sigma|}{|\sigma|}\inner{\mathsf{u}}{\sigma}\inner{\mathsf{v}}{\sigma}, 
  \end{gather}
  where $\inner{\mathsf{u}}{\sigma} \equiv \mathsf{u}(\sigma)$ is the duality pairing of
  $C^{k}(\mathcal{T})$ and $\mathcal{T}^{k}$. We denote the
  norm induced by this inner product by $\normm{\cdot}_k$. Let
  $C(\mathcal{T}) := \bigoplus_k C^k(\mathcal{T})$ and define an inner product thereon
  for $\mathbf{u}, \mathbf{v}\in C(\mathcal{T})$ by
  \begin{gather*}\innerc{\mathbf{u}}{\mathbf{v}} := \sum_{k = 0}^n \innerc{\mathsf{u}_k}{\mathsf{v}_k}_k, \end{gather*} where
  $\mathbf{u} = (\mathsf{u}_0, \dots, \mathsf{u}_n), \mathsf{u}_k\in C^k(\mathcal{T})$. For the induced norm we
  write $\normm{\cdot}$.
\end{definition}

\begin{definition}[{Discrete Codifferential, \cite[Definition 3.7, Definition 4.3]{GUP25}}]
  The discrete codifferential
  $\delta^\textnormal{DEC}_{k+1}: C^{k+1}(\mathcal{T})\to C^k(\mathcal{T})$ is defined
  as
  \begin{gather*}
    \delta^\textnormal{DEC}_{k + 1} := (-1)^{k+1} \ccstar_{k}^{-1} \tilde{\partial}^{n - (k + 1)}\ccstar_{k + 1}, \quad k = 0, \dots, n - 1,
  \end{gather*} where $\tilde{\partial}^{n - (k + 1)}$ denotes the coboundary operator on
  $C^{n - (k + 1)}(\tilde{\mathcal{T}})$.
\end{definition}
In analogy to $\dd$ and $\delta$ from \eqref{ddd} we set
\begin{gather*}
  \DECd :=
  \begin{pmatrix} 
    0 &  & \\
    \partial^0 & 0 & \\
    & \partial^1 & 0 & \\
    && \ddots & \ddots
  \end{pmatrix}, \quad
  \DECdelta :=
  \begin{pmatrix} 
    0 & \DECdelta_1 & \\
    & 0 & \DECdelta_2 \\
    & & 0 & \ddots \\
    && & \ddots 
  \end{pmatrix} ,
\end{gather*}
the exterior derivative and co-derivative on the exterior algebra of cochains mapping
$C(\mathcal{T})\to C(\mathcal{T})$.

\begin{lemma}[Discrete Adjoint]\label{lemma:adj} It holds that
  \begin{gather*}
    \innerc{\DECdelta \mathbf{u}}{\mathbf{v}} = \innerc{\mathbf{u}}{\DECd \mathbf{v}} \quad \forall \mathbf{u}, \mathbf{v} \in C(\mathcal{T}).
  \end{gather*}
\end{lemma}
\begin{proof}
  This follows directly from \cite[Lemma 4.12]{GUP25} and the definition of $\innerc{\cdot}{\cdot}$.
\end{proof}
\begin{definition}\label{def:dec_dirac}
  The DEC Hodge-Dirac operator is $\DECD := \DECd + \DECdelta.$
\end{definition}

\begin{definition}
  Define
  $\mathring{C}(\mathcal{T}) := \bigoplus_{k = 0}^{n - 1}\mathring{C}^k(\mathcal{T})
  \oplus C_*^n(\mathcal{T})$, where $\mathring{C}^k(\mathcal{T})$ are $k$-cochains
  with zero values on the boundary and $C_*^n(\mathcal{T})$ is the space of
  $n$-cochains with vanishing mean, i.e.\
  $w\in C_*^n(\mathcal{T})\implies \sum_{\sigma\in\mathcal{T}^n}\inner{w}{\sigma} = 0$.
\end{definition}

\section{Commuting Interpolation Operators}

Let $\ipop^k$ denote the canonical projection onto $C^k(\mathcal{T})$, the \emph{de Rham map}, defined
for sufficiently smooth forms that admit an $L^1$ trace on $k$-simplices, see
\cite[Section~4.4]{GUP25}, and set $\ipop := \bigoplus_k \ipop^k$. 
% \rev{ Note that in \cite{GUP25}, $\ipop^k$ is used for the
% canonical projection onto Whitney forms and not the de Rham map. 
% TODO: Maybe switch $\ipop^k$ and $\mathcal{R}^k$?}
From the results in \cite[Section 4]{GUP25} we learn the following commuting diagram property. 

\begin{lemma}\label{lemma:pid}
  On sufficiently smooth forms, we have the commuting relationship
  \begin{gather*}
    \DECd\circ\ipop = \ipop\circ\dd.
  \end{gather*}
\end{lemma}

\begin{definition}[{\cite[Section~5.1]{GUP25}}]
    Let $J_k := \ccstar_{k}^{-1}\tilde{\ipop}^{n - k}\star_k$, where
    $\tilde{\ipop}^k$ denotes the canonical projection onto $C^k(\tilde{\mathcal{T}})$, and define $J := \bigoplus_k J_k$.
  \end{definition}
  
\rev{
\begin{remark}
    Note that the operator $J$ in \cite{GUP25} is defined on Whitney forms, whereas the $J$ considered here is defined on cochains. The two versions are related via the Whitney isomorphism, which allows one to recover either from the other.
\end{remark}
}
\rev{Appealing to \cite[Lemma~5.3]{GUP25}, the following commuting property follows
  immediately.}
\begin{lemma}\label{lemma:jdelta}
  On sufficiently smooth forms, we have the commuting relationship $$\DECdelta\circ J = J\circ\delta. $$
\end{lemma}
The natural DEC discretization of \eqref{eq:strong_form} is: Given
$\mathfrak{f}\in L^2\LambdaForms$ with $\int_\Omega f_n = 0$ which is sufficiently regular to admit
an $L^1$-trace on all simplices, seek $\mathbf{u}\in\mathring{C}(\mathcal{T})$, such that
\begin{gather}
  \label{eq:Dbvp}
  \DECD \mathbf{u} = \ipop\mathfrak{f}.
\end{gather} We need $\int_\Omega f_n = 0$ also in the discrete case, because we want
$\ipop\mathfrak{f}\in\ran\DECD$, but $\ran\DECD$ only contains $n$-cochains with zero
mean. The de Rham map $\ipop$ preserves the integral of the traces over the simplices,
so we get $\ipop\mathfrak{f}\in\ran\DECD$.

\section{A Priori Discretization Error Estimate}

As before, we consider the boundary value problem for the Hodge-Dirac operator with
essential boundary conditions, so assume that $\DD$ acts on spaces with \emph{zero trace}
or \emph{zero mean} in the case of $n$-forms. This also means that
$\mathfrak{f}\in\ran\DD\implies \int_\Omega {f}_n = 0$.

New in this section, now we regard $\mathcal{T}$ as a member of a \emph{uniformly
  shape-regular} sequence $\left( \mathcal{T}_{h}\right)$ of simplicial meshes of $\Omega$
indexed by their meshwidths $h$, which \rev{ tend to zero}.

\subsection{An h-Uniformly Stable Decomposition}
Let $V_h^k\subset H\LambdaForms[k]$ denote the finite element spaces of lowest order
discrete differential forms on $\mathcal{T}$, known as Whitney forms, see \cite[Section
4]{GUP25}, and $\mathring{V}_h := \mathring{V} \cap \bigoplus_{k = 0}^n V_h^k$. Note that
$\mathring{V}_h$ contains only Whitney forms with zero trace on $\partial\Omega$ or zero
mean in the case of $n$-forms. We point out that the FEEC approach to \eqref{eq:3} from
\cite{Leopardi2016} employs $V_{h}^{k}$ for the Galerkin discretization of
\eqref{eq:3}. Here, we will need Whitney forms only as a theoretical tool, exploiting the
algebraic isomorphism of $C^k(\mathcal{T})$ and $V_{h}^{k}$ via the canonical degrees of
freedom.

\begin{definition}[{Whitney Map, \cite[Section~4.4]{GUP25}}]
  The Whitney map on $k$-cochains is the isomorphism onto
  Whitney $k$-forms $\mathcal{W}^k: C^k(\mathcal{T}) \to V_h^k$, given by
  $\mathcal{W}^k\mathsf{w} = \sum_{\sigma\in\mathcal{T}^k}\phi_\sigma\inner{\mathsf{w}}{\sigma}$ where
  $\mathsf{w}\in C^k(\mathcal{T})$ and $\phi_\sigma$ is the Whitney basis form associated to
  $\sigma$. Also, let $\mathcal{W} := \bigoplus_{k = 0}^n\mathcal{W}^k$.
\end{definition}
Note that $\mathcal{W}^{k}$ is represented by an identity matrix with respect to the standard bases of $C^k(\mathcal{T})$ and $V_h^k$.

\begin{lemma}[{\cite[Lemma 4.11]{GUP25}}]\label{lemma:norm_eq}
  The norms $\normm{\cdot}$ and $\norm{\cdot}_{L^2\LambdaForms}$ on spaces of cochains
  and Whitney forms are $h$-uniformly equivalent via the $\mathcal{W}$-isomorphism. More
  precisely, there exist constants $c_-, c_+ > 0$ depending only on the shape-regularity
  of \rev{both meshes $\mathcal{T}$ and $\tilde{\mathcal{T}}$}, such that
    \[
        c_- \norm{\mathcal{W}\mathbf{u}}_{L^2\LambdaForms} \leq \normm{\mathbf{u}} \leq c_+ \norm{\mathcal{W}\mathbf{u}}_{L^2\LambdaForms}\quad\forall \mathbf{u}\in C(\mathcal{T}).
    \]
\end{lemma}

Similar to \cite[Lemma 5.11]{GUP25}, we will need the following result relating to the
Hodge-decomposition:
\begin{lemma}\label{lemma:hodge_decomp}
  For any $\mathbf{u}\in\mathring{C}(\mathcal{T})$ there exist
  $\mathbf{v}, \mathbf{w} \in\mathring{C}(\mathcal{T})$ such that
  \begin{align*}
    \mathbf{u} & = \DECd\mathbf{v} + \mathbf{w}, \\
    \normm{\mathbf{w}} + \normm{\DECd\mathbf{w}} &\leq C \normm{\DECd\mathbf{u}}, \\
    \normm{\mathbf{v}} + \normm{\DECd \mathbf{v}} &\leq C' \normm{\mathbf{u}}
    \end{align*} for constants $C, C' \geq 0$ independent of $\mathbf{u}$ and the
    meshwidth $h$.
\end{lemma}
\begin{proof}
    First, we note that due to the norm equivalence from \Cref{lemma:norm_eq} and the fact that $\mathcal{W}\DECd \equiv \dd\mathcal{W}$, it is sufficient to prove that $\exists \alpha, \beta\in \mathring{V}_h$ such that
    \begin{align}
        \mathcal{W}\mathbf{u} & = \dd\alpha + \beta, \label{eq:form_decomp} \\
        \norm{\beta}_{L^2\LambdaForms} + \norm{\dd\beta}_{L^2\LambdaForms} &\leq C \norm{\dd\mathcal{W}\mathbf{u}}_{L^2\LambdaForms} \label{eq:beta_ineq}, \\
        \norm{\alpha}_{L^2\LambdaForms} + \norm{\dd\alpha}_{L^2\LambdaForms} &\leq C' \norm{\mathcal{W}\mathbf{u}}_{L^2\LambdaForms}, \label{eq:alpha_ineq}.
      \end{align} To see this, we consider the discrete Hodge decomposition (see
      \cite[Section 3.1]{Leopardi2016}; note that $\mathring{V}_h$ does not contain harmonic forms \rev{as $\Omega$ is topologically trivial by assumption.})
    \[
        \mathring{V}_h = \mathfrak{B}_h \oplus\mathfrak{Z}_h^\perp,
    \] where $\mathfrak{B}_h$ is the range and $\mathfrak{Z}_h$ the kernel of $\dd|_{\mathring{V}_h}$. Moreover, this decomposition is $L^2\LambdaForms$-orthogonal. \\
    This implies that we can find $\alpha\in\mathring{V}_h, \beta\in\mathfrak{Z}_h^\perp$ such that $\mathcal{W}\mathbf{u} = \dd\alpha + \beta$. Note that $\alpha$ is not unique, as adding any element in $\mathfrak{Z}_h$ gives the same $\dd\alpha$, hence we can safely assume that there exists an $\alpha$ orthogonal to $\mathfrak{Z}_h$, meaning we can find suitable $\alpha, \beta\in\mathfrak{Z}_h^\perp$.
    
    We can now apply the discrete Poincar\'e inequality from \cite[Lemma 9]{Leopardi2016} (see also \cite[Theorem 5.2]{Arnold2018}) to $\beta$, which is applicable as $\beta\in\mathfrak{Z}_h^\perp$, and get
    \begin{align*}
        \norm{\beta}_{L^2\LambdaForms} + \norm{\dd\beta}_{L^2\LambdaForms} &\leq C \norm{\dd\beta}_{L^2\LambdaForms} & \left[\text{Poincar\'e Inequality}\right] \\
        & = C \norm{\dd\left(\mathcal{W}\mathbf{u} - \dd\alpha\right)}_{L^2\LambdaForms} & \left[\text{Using \eqref{eq:form_decomp}}\right] \\
        & = C \norm{\dd\mathcal{W}\mathbf{u}}_{L^2\LambdaForms} & \left[\dd^2 \equiv 0\right]
    \end{align*} for some mesh-width independent $C \geq 0$, which proves \eqref{eq:beta_ineq}. \\
    To prove \eqref{eq:alpha_ineq}, we can apply the Poincar\'e inequality to
    $\alpha\in\mathfrak{Z}_h^\perp$ and use the orthogonality of the decomposition to
    arrive at 
    \begin{align*}
        \norm{\alpha}_{L^2\LambdaForms}^2 + \norm{\dd\alpha}_{L^2\LambdaForms}^2 &\leq C' \norm{\dd\alpha}_{L^2\LambdaForms}^2 \leq C' \left(\norm{\dd\alpha}_{L^2\LambdaForms}^2 +
          \norm{\beta}_{L^2\LambdaForms}^2\right)
        = C' \norm{\mathcal{W}\mathbf{u}}_{L^2\LambdaForms}^2.
    \end{align*} Applying Young's inequality to the above concludes the proof.
  \end{proof}
  
  \subsection{Main Error Bound}
  
% \begin{tcolorbox}
\noindent\fbox{\parbox{\linewidth}{
    \begin{theorem}\label{theorem:norme}
      Given $\mathfrak{f}\in\ran\DD$ in the domain of $\ipop$ and $J$, we assume
      that the strong solution $\mathfrak{u}\in \mathring{V}\cap H^*\LambdaForms$ of
      the Hodge-Dirac boundary value problem
      \begin{gather*}
        \DD\mathfrak{u} = \mathfrak{f}
      \end{gather*}
      is sufficiently regular such that \Cref{lemma:pid} and \Cref{lemma:jdelta} apply and
      $\ipop$ and $J$ are well-defined on $\mathfrak{u}, \dd\mathfrak{u}$ and
      $\delta\mathfrak{u}$. Further let $\mathbf{u}\in \mathring{C}(\mathcal{T})$ solve
      the DEC equation
      \begin{gather*}
        \DECD \mathbf{u} = \ipop\mathfrak{f},
      \end{gather*}
      and denote the error in cochain space by $e := \ipop\mathfrak{u} - \mathbf{u}$. Then
      \begin{align*}
        \normm{e} + \normm{\DECd e} \leq& C\left(\normm{(\ipop - J) \mathfrak{u}} + \normm{(\ipop - J) \dd\mathfrak{u}} + \normm{(\ipop - J)\mathfrak{f}}\right),
      \end{align*} where $C \geq 0$ is a constant independent of $\mathfrak{f}$ and the
      meshwidth $h$. 
    \end{theorem}}}

% \end{tcolorbox}
\begin{proof}[Proof of \Cref{theorem:norme}] Unless stated otherwise, $C, C' \geq 0$
  denote generic mesh-width independent constants which may change from expression to
  expression.
  
  Similar to the proof of Theorem 5.2 in \cite{GUP25}, we apply the operator to the
  error $e$ and get an estimate of $\normm{\DECd e}$:
  \begin{flalign*}
     & & \DECD e & = (\DECd + \DECdelta)e & \left[\text{\Cref{def:dec_dirac}}\right]\\
    & & & = \DECd\ipop\mathfrak{u} + \DECdelta\ipop\mathfrak{u} - (\DECd +
    \DECdelta)\mathbf{u} & \left[e = \ipop\mathfrak{u} - \mathbf{u}\right] \\ & & & =
    \DECd\ipop\mathfrak{u} + \DECdelta\ipop\mathfrak{u} - \ipop\mathfrak{f} &
    \left[\text{Using } (\DECd + \DECdelta)\mathbf{u} = \ipop\mathfrak{f}\right] \\ & &
    & = \ipop\dd\mathfrak{u} + \DECdelta J \mathfrak{u} + \DECdelta(\ipop -
    J)\mathfrak{u} - \ipop f & \left[\text{Adding $0$, \Cref{lemma:pid}}\right] \\ & &
    & = (\ipop - J)\dd\mathfrak{u} + \DECdelta (\ipop - J)\mathfrak{u} + (J -
    \ipop)\mathfrak{f}. & \hspace*{-3em}
    \left[
      \text{Using }\DECdelta J\mathfrak{u} = J\delta\mathfrak{u}
      = J(\mathfrak{f} - \dd\mathfrak{u})\right]
      \stepcounter{equation}\tag{\theequation}\label{eq:dec_dirac_e}
  \end{flalign*}
  \rev{
  Moreover,
  \begin{flalign*}
    && \DECdelta\DECd e & = \DECdelta(\DECd + \DECdelta) e &
    \left[\text{Using }\left(\DECdelta\right)^2 = 0\right] \\ & &
    & = \DECdelta (\ipop - J)\dd\mathfrak{u} + \DECdelta(J - \ipop)\mathfrak{f}, & \left[\text{Using \eqref{eq:dec_dirac_e}}\right]
  \end{flalign*}
  and
  \begin{flalign*}
    & & \normm{\DECd e}^2 & = \innerc{\DECd e}{\DECd e} = \innerc{\DECdelta
      \DECd e}{e} & \left[\text{\Cref{lemma:adj}}\right] \\ & & & =
    \innerc{\DECdelta(\ipop - J)\dd\mathfrak{u} + \DECdelta(J - \ipop)\mathfrak{f}}{e} &
    \left[\text{Using } \left(\DECdelta\right)^2 = 0\right] \\ && & = \innerc{(\ipop -
      J)\dd\mathfrak{u}}{\DECd e} + \innerc{(J - \ipop)\mathfrak{f}}{\DECd e} &
    \left[\text{\Cref{lemma:adj}}\right] \\ &&
    & \leq \normm{(\ipop - J)\dd\mathfrak{u}}\normm{\DECd e} + \normm{(J - \ipop)\mathfrak{f}}\normm{\DECd e}. & \left[\text{Cauchy-Schwarz}\right] \\
    \implies & & \normm{\DECd e} &\leq \normm{(\ipop - J)\dd\mathfrak{u}} + \normm{(\ipop -
      J)\mathfrak{f}}.& \stepcounter{equation}\tag{\theequation}\label{eq:de_estimate}
  \end{flalign*}
  }
  To bound $\normm{e}$, we proceed similarly to \cite[Lemma 5.14]{GUP25}: Using \Cref{lemma:hodge_decomp}, we find $\mathbf{v}, \mathbf{w}\in\mathring{C}(\mathcal{T})$ such that
  \begin{align*}
    e = \DECd\mathbf{v} + \mathbf{w}, \quad
    \normm{\mathbf{v}} + \normm{\DECd \mathbf{v}} \leq C \normm{e}, \quad
    \stepcounter{equation}\tag{\theequation}\label{eq:v_estimate}
    \normm{\mathbf{w}} \leq C' \normm{\DECd e}.
  \end{align*}
  Thus,
  \[
    \normm{e}^2 = \innerc{\DECd\mathbf{v}}{e} + \innerc{\mathbf{w}}{e} = \innerc{\mathbf{v}}{\DECdelta e} + \innerc{\mathbf{w}}{e}
  \] by \Cref{lemma:adj}. We can immediately estimate the second term \rev{using \eqref{eq:de_estimate} and \eqref{eq:v_estimate}}:
  \begin{equation}
      \innerc{\mathbf{w}}{e} \leq \normm{\mathbf{w}}\normm{e} \leq C' \normm{\DECd e}\normm{e} \leq C' \left(\normm{(\ipop - J)\dd\mathfrak{u}} + \normm{(\ipop -
      J)\mathfrak{f}}\right) \normm{e}. \label{eq:pibar_delta_varphi}
  \end{equation}
  In order to estimate the first term, we re-write $\DECdelta e$ using \eqref{eq:dec_dirac_e}:
  \[
      \DECdelta e = \DECD e - \DECd e = \DECdelta(\ipop - J)\mathfrak{u} + (\ipop - J)\dd\mathfrak{u} + (J - \ipop)\mathfrak{f} - \DECd e.
  \]
  \begin{flalign*}
      \implies \innerc{\mathbf{v}}{\DECdelta e} & = \innerc{\mathbf{v}}{\DECdelta(\ipop - J)\mathfrak{u} + (\ipop - J)\dd\mathfrak{u} + (J - \ipop)\mathfrak{f} - \DECd e} & \\
      & = \innerc{\DECd\mathbf{v}}{(\ipop - J)\mathfrak{u}} + \innerc{\mathbf{v}}{(\ipop - J)\dd\mathfrak{u}} + \\
      & \ \quad \innerc{\mathbf{v}}{(J - \ipop)\mathfrak{f}} - \innerc{\mathbf{v}}{\DECd e} & \\
      & \leq \normm{\DECd\mathbf{v}}\normm{(\ipop - J) \mathfrak{u}} + \\
      & \ \quad \normm{\mathbf{v}}\left(\normm{(\ipop - J) \dd\mathfrak{u}} + \normm{(\ipop - J)\mathfrak{f}} + \normm{\DECd e}\right) &\left[\text{Cauchy-Schwarz}\right] \\
      & \leq \normm{\DECd\mathbf{v}}\normm{(\ipop - J) \mathfrak{u}} + \\
      & \ \quad 2\normm{\mathbf{v}}\left(\normm{(\ipop - J) \dd\mathfrak{u}} + \normm{(\ipop - J)\mathfrak{f}}\right) & \left[\text{Using \eqref{eq:de_estimate}}\right] \\
      & \leq 2 \left(\normm{(\ipop - J) \mathfrak{u}} + \normm{(\ipop - J) \dd\mathfrak{u}} + \right. \\
      & \ \qquad \left. \normm{(\ipop - J)\mathfrak{f}}\right)\left( \normm{\mathbf{v}} + \normm{\DECd\mathbf{v}} \right) \\
      & \leq C \left(\normm{(\ipop - J) \mathfrak{u}} + \normm{(\ipop - J) \dd\mathfrak{u}} + \normm{(\ipop - J)\mathfrak{f}}\right)\normm{e}. & \left[\text{Using \eqref{eq:v_estimate}}\right]
      \stepcounter{equation}\tag{\theequation}\label{eq:pieta_deltae}
  \end{flalign*} Combining \eqref{eq:pibar_delta_varphi} and \eqref{eq:pieta_deltae}, we get
  \begin{flalign*}
      & &\normm{e}^2 = & \innerc{\mathbf{v}}{\DECdelta e} + \innerc{\mathbf{w}}{e} & \\
      & & \leq & C' \left(\normm{(\ipop - J)\dd\mathfrak{u}} + \normm{(\ipop -
      J)\mathfrak{f}}\right) \normm{e} + & \\
      & & & C\ \left(\normm{(\ipop - J) \mathfrak{u}} + \normm{(\ipop - J) \dd\mathfrak{u}} + \normm{(\ipop - J)\mathfrak{f}}\right)\normm{e}. &
  \end{flalign*}
  \rev{
  Employing \eqref{eq:de_estimate}, we conclude that
  \begin{flalign*}
      & & \normm{e} + \normm{\DECd e} \leq& C\left(\normm{(\ipop - J) \mathfrak{u}} + \normm{(\ipop - J) \dd\mathfrak{u}} + \normm{(\ipop - J)\mathfrak{f}}\right), &
  \end{flalign*} which is the assertion of the theorem.
  }
\end{proof}

\begin{remark}
  We examined the discretization error in a finite-difference sense as the difference of
  the discrete solution and a ``projection of the exact solution on $\mathcal{T}$''. We
  can easily obtain an error estimate in the FEEC sense:
    \[
        \norm{\mathfrak{u} - \mathcal{W}\mathbf{u}}_{H\Lambda(\Omega)} \leq C\left(\norm{\mathfrak{u} - \mathcal{W}\ipop\mathfrak{u}}_{H\Lambda(\Omega)} + \normm{e} + \normm{\DECd e}\right).
    \]
\end{remark}

\subsection{Rates of Convergence}
Now that we can bound the discrete error, we only need an estimate for $\ipop - J$,
which is given in \cite{GUP25}.

\begin{lemma}\label{lemma:pij_difference}
  Given a sufficiently smooth $\mathfrak{u} \equiv (u_0, \dots, u_n)$ in the exterior
  algebra of differential forms, we have
  \begin{gather*}
    \normm{(\ipop - J)\mathfrak{u}}^2 \leq C \sum_{k = 0}^n
    \sum_{s = 1}^{r_k} h^{2s}|u_k|^2_{H^s(\Omega)},
  \end{gather*}
  where $h$ is the mesh-width, $C \geq 0$ a constant independent of $h$ and
  $r_k = \max\left\{ \lceil \frac{n - k}{2} + \varepsilon \rceil, \lceil \frac{k}{2}
    + \varepsilon \rceil \right\}$ for any $0 < \varepsilon < 1$.
\end{lemma}
\begin{proof}
  \cite[Lemma 5.10]{GUP25} tells us that for all $k$, we have
  \begin{gather*}
    \normm{(\ipop^k - J_k)u_k}_k^2 \leq C_k \sum_{s = 1}^{r_k} h^{2s}|u_k|^2_{H^s(\Omega)}
  \end{gather*}
  for some $h$-independent constant $C_k$. Realizing that
  $\normm{\mathbf{v}}^2 = \sum_{k = 0}^n \normm{\mathsf{v}_k}^2_k$ for all
  $\mathbf{v}\in C(\mathcal{T})$ and setting $C = \max_k C_k$ yields the desired
  result.
\end{proof}

In a similar setting as before, we can prove an estimate for sufficiently smooth
solutions. Let $C^\ell\LambdaForms$ denote the space of $\ell$-times continuously
differentiable forms.

\begin{proposition}
  \label{thm:43}
  Let $r := \lceil \frac{n}{2} + \epsilon \rceil$ for any $0 < \epsilon < 1$. Given
  $\mathfrak{f}\in\ran\DD \cap C^r\LambdaForms$, assume that we are given a strong
  solution $\mathfrak{u}\in \mathring{V}\cap C^{r + 1}\LambdaForms$ to the Hodge-Dirac
  problem
  \begin{gather*}
    \DD\mathfrak{u} = \mathfrak{f}.
  \end{gather*} \
  Let $\mathbf{u}\in \mathring{C}(\mathcal{T})$ solve its discrete counterpart
  \begin{gather*}
    \DECD \mathbf{u} = \ipop\mathfrak{f},
  \end{gather*}
  and denote the error in cochain space by $e := \ipop\mathfrak{u} - \mathbf{u}$. Then
  \begin{gather*}
    \normm{e}^2 + \normm{\DECd e}^2 \leq C \sum_{k = 0}^n \sum_{s = 1}^{r_k} h^{2s}\left(|u_k|^2_{H^s(\Omega)} + \left|\left(\dd\mathfrak{u}\right)_k\right|^2_{H^s(\Omega)} + |f_k|^2_{H^s(\Omega)} \right)
  \end{gather*} for some $C \geq 0$ independent of $h$ and $r_k = \max\left\{ \lceil \frac{n - k}{2} + \varepsilon \rceil, \lceil \frac{k}{2} + \varepsilon \rceil \right\}$ for any $0 < \varepsilon < 1$.
\end{proposition}
\begin{proof} 
    Let $p = \normm{e}, q = \normm{\DECd e}, r = \normm{(\ipop - J) \mathfrak{u}}, s = \normm{(\ipop - J) \dd\mathfrak{u}}, t = \normm{(\ipop - J)\mathfrak{f}}$, then the estimate in \Cref{theorem:norme} says
    \[
        p + q \leq C(r + s + t).
    \] By (repeated application of) Young's inequality and because $p, q \geq 0$, we have
    \begin{align*}
        p^2 + q^2 &\leq (p + q)^2 \leq C^2(r + s + t)^2 \leq 2C^2(r^2 + (s + t)^2) \leq 4C^2(r^2 + s^2 + t^2).
    \end{align*}
    The statement then follows from applying \Cref{lemma:pij_difference} to the terms $r^2, t^2$ and $s^2$.
\end{proof}

\begin{remark}
  The approach pursued in the present paper follows \cite{GUP25} very closely, but a different route
  could have been taken to prove convergence, at least in 2D. \cite{Zhu2025} uses a
  very close relationship between the inner products from FEEC and DEC (on suitable
  meshes) to show that the consistency gap between FEEC and DEC solutions for the
  Hodge-Laplacian decreases with the mesh-width, then convergence of FEEC implies
  convergence of DEC. The advantage of that approach is that we do not have to assume
  that the solution enjoys the high regularity stipulated in \Cref{thm:43}.
\end{remark}

\section{Numerical Tests in Two Dimensions}

We employ the method of manufactured solutions to empirically verify the order of
convergence obtained in \Cref{thm:43}. We measure DEC norms of the discretization error
$e:=\ipop\mathfrak{u}-\mathbf{u}$, $\mathbf{u}$ solution of \eqref{eq:Dbvp}, where
$\ipop\mathfrak{u}$ and $\ipop\mathfrak{f}$ are computed by ``overkill quadrature'', which
means that the quadrature error is negligible compared to the discretization error. We
monitor two error norms: When we talk about the DEC $L^2$-norm we mean $\normm{e}$, and by
the DEC $H\Lambda$-norm we mean $\normm{e} + \normm{\DECd e}$.

The implementation of the DEC scheme relied on the C++ finite element library
\texttt{MFEM} \cite{mfem}. The concrete code used for carrying out the tests can be
found at \url{https://github.com/rdabetic/2d_dec_dirac}.

\subsection{Test I}
We consider the unit square $\Omega = [0, 1]^2$ and fix the right-hand-side $\mathfrak{f}$
such that we obtain a smooth solution of
\eqref{eq:strong_form}, which reads 
\[
  u_0 = \sin 2\pi x \sin 2 \pi y, \quad u_1 = (\sin 2\pi y, \sin 2\pi x)^T, \quad u_2 = \cos 2\pi x \sin 2 \pi y
\] in Euclidean vector proxies. \\
The coarsest mesh that was used is displayed in \Cref{fig:square_msh}. It was refined
several times using regular refinement, i.e.\ connecting the midpoints of the edges to
split each triangle into four smaller ones.

The resulting error norms are plotted in \Cref{fig:square_cvg}, and we observe first-order
convergence, exactly the order of convergence predicted by \Cref{thm:43}.

\begin{figure}[h]
    \centering
    \begin{subfigure}{0.49\textwidth}
        \includegraphics[width=\textwidth]{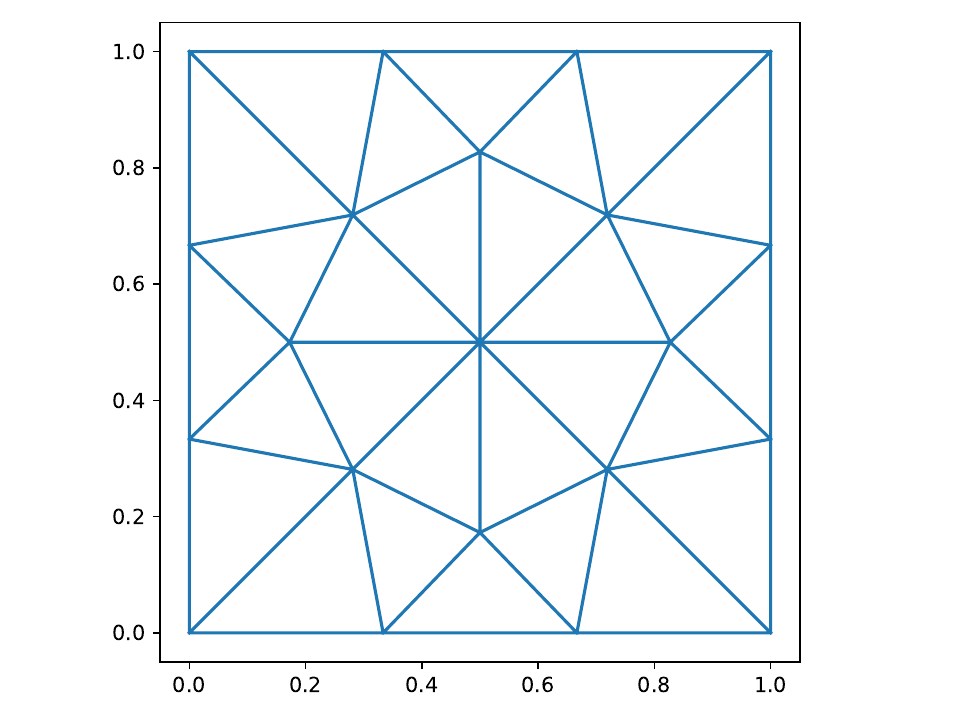}
        \caption{Test I: Coarsest mesh}
        \label{fig:square_msh}
    \end{subfigure}
    \begin{subfigure}{0.49\textwidth}
        \includegraphics[width=\textwidth]{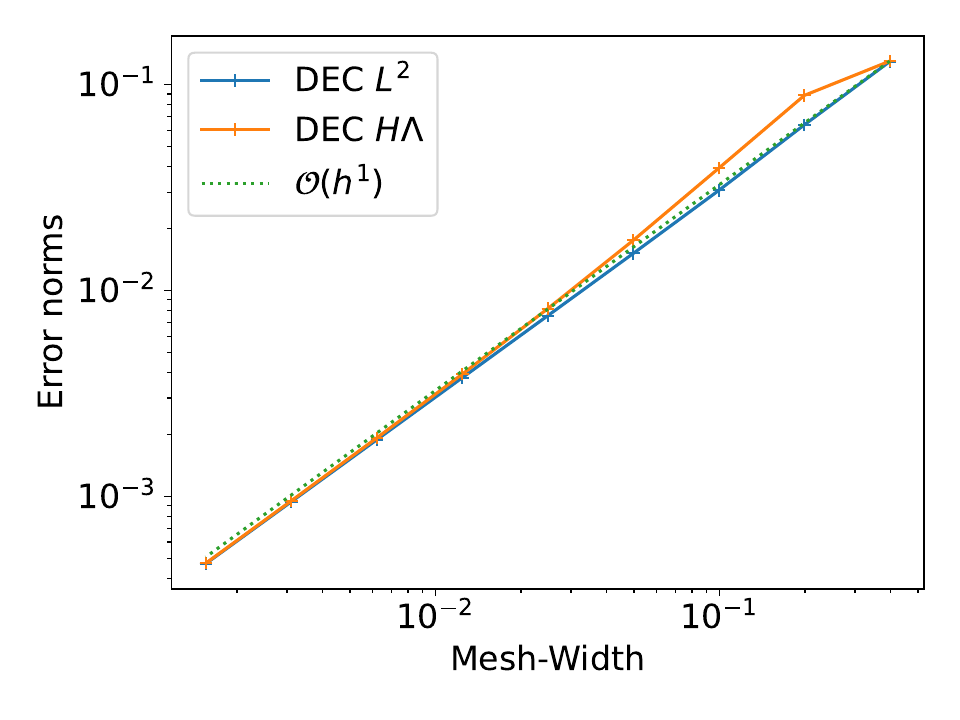}
        \caption{Test I: Error norms}
        \label{fig:square_cvg}
    \end{subfigure}
    \caption{Mesh and convergence of DEC on a square.}
\end{figure}
\subsection{Test II}
Similar to \cite{GUP25}, we tested the DEC discretization on a triangle as well. As a
domain $\Omega\subset\R^2$ we chose an equilateral triangle with vertices at $(0, 0)$,
$(0, 1)$, and $\left(\nicefrac{1}{2}, \nicefrac{\sqrt 3}{2}\right)$.
We fix the right-hand-side such that we obtain the exact solution (in Euclidean vector proxies) 
\[
    u_0 = 2^{15}\left(\lambda_0\lambda_1\lambda_2\right)^3, \quad u_1 = (u_0, u_0)^T, \quad u_2 = u_0 - \frac{1}{|\Omega|}\int_\Omega u_0(x, y)\ \dd x \dd y,
  \] where $\lambda_i$ denotes the barycentric coordinate function associated with vertex $i$. \\
  As before, we used successive regular refinement of a coarse mesh, which can be seen in
  \Cref{fig:tria_mesh}, to generate a sequence of meshes with decreasing mesh-width. Note
  that the refined meshes only contain equilateral triangles.

  The plot of \Cref{fig:tria_cvg} clearly reveals that for $h \to 0$ the error norms
  decrease faster than expected. The better-than-expected order of convergence is most likely due to the
  symmetry of the mesh (all equilateral triangles), as explained in \cite[Section
  6]{GUP25}, where the authors provide improved error estimates on $\ipop - J$ in such a
  case. Concretely, \cite[Equation 6.2 \& Proposition 6.2]{GUP25} establish second order
  convergence, which is what is observed.
\begin{figure}[h]
    \centering
    \begin{subfigure}{0.49\textwidth}
        \includegraphics[width=\textwidth]{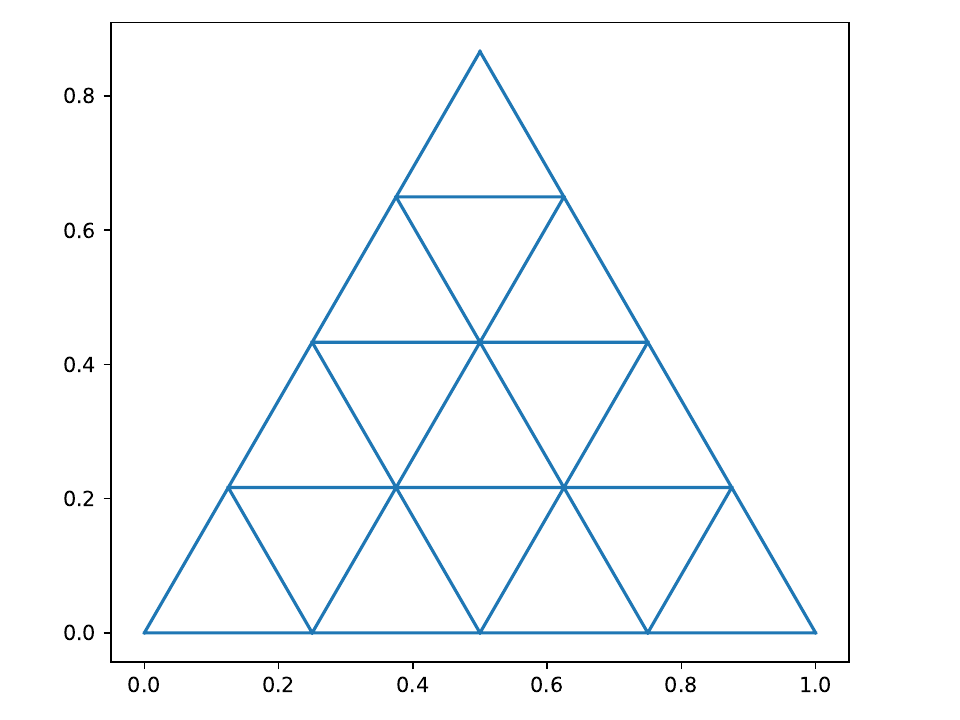}
        \caption{Test II: Coarsest mesh}
        \label{fig:tria_mesh}
    \end{subfigure}
    \begin{subfigure}{0.49\textwidth}
        \includegraphics[width=\textwidth]{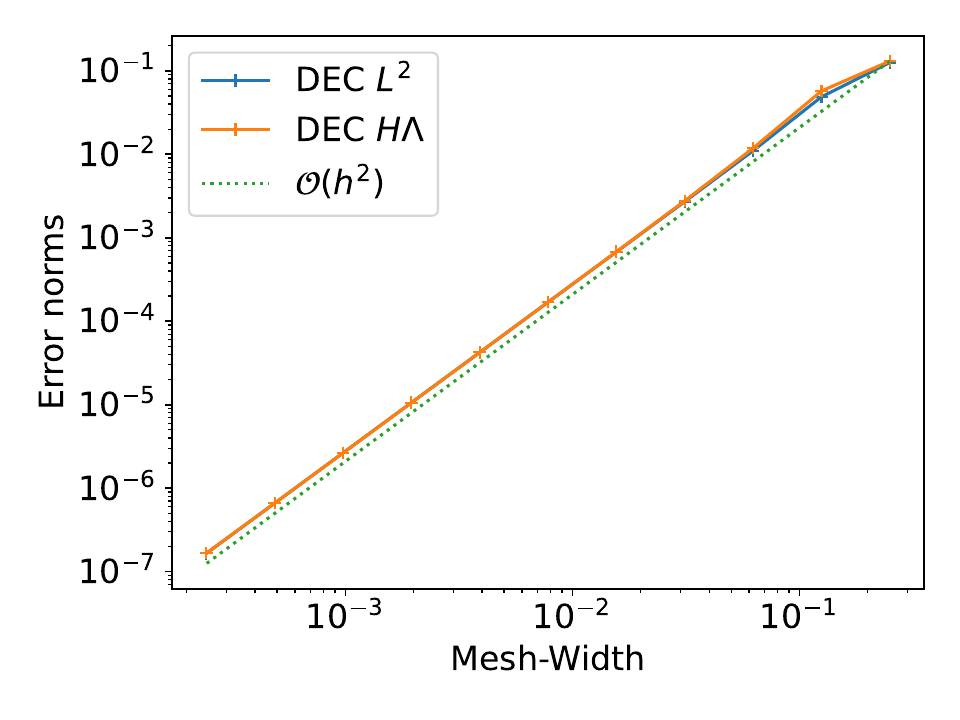}
        \caption{Test II: Error norms}
        \label{fig:tria_cvg}
    \end{subfigure}
    \caption{Mesh and convergence of DEC on an equilateral triangle with a structured mesh.}
\end{figure}

\subsection{Test III}

\begin{figure}[H]
    \centering
    \begin{subfigure}{0.49\textwidth}
        \includegraphics[width=\textwidth]{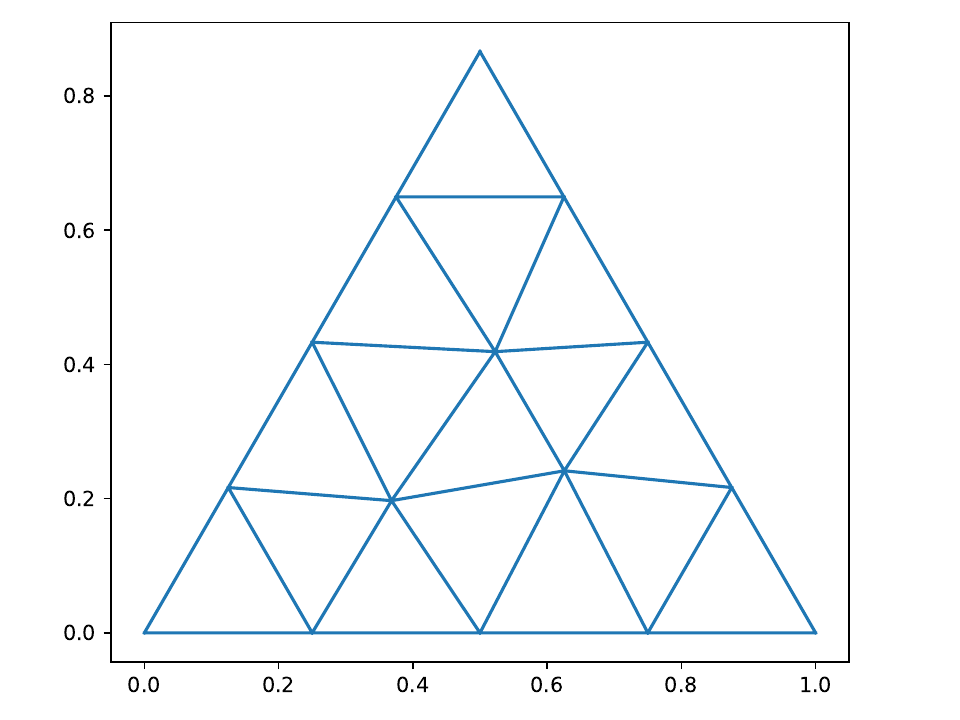}
        \caption{Test III: Coarsest mesh}
        \label{fig:tria_pert_mesh}
    \end{subfigure}
    \begin{subfigure}{0.49\textwidth}
        \includegraphics[width=\textwidth]{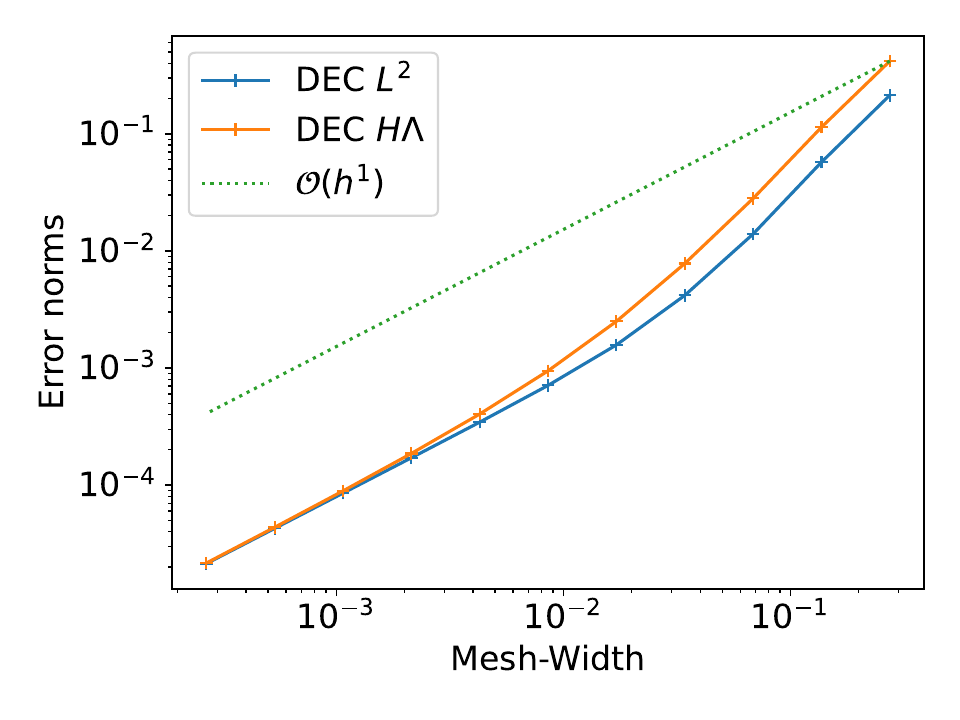}
        \caption{Test III: Error norms}
        \label{fig:tria_cvg_pert}
    \end{subfigure}
    \caption{Mesh and convergence of DEC on an equilateral triangle with a perturbed mesh.}
  \end{figure}
  
  This numerical experiment is inspired by the "Perturbed Mesh" computation in \cite[Section~7]{GUP25}. The setup is the same as in Test II, but now we start with a
  slightly perturbed coarse mesh of the triangle domain $\Omega$, see
  \Cref{fig:tria_pert_mesh}.  This breaks symmetries, the theory from
  \cite[Section~6]{GUP25} no longer applies and, as one can see from
  \Cref{fig:tria_cvg_pert}, now convergence of error norms appears to be first order,
  albeit with some pre-asymptotic behavior.
  
\section*{Declarations}
The authors have no conflict of interest to declare that are relevant to the content of this article.

% \printbibliography
\bibliographystyle{plainnat}
\bibliography{sources}
\end{document}